\documentclass[12pt]{article}
\usepackage{amssymb}
\usepackage{amsthm}
\usepackage{amsmath}
\usepackage{graphicx}
\usepackage{enumerate}
\usepackage{pict2e}

\DeclareMathOperator{\Con}{Con}

\DeclareMathOperator{\Max}{Max}

\newtheorem{theorem}{Theorem}[section]
\newtheorem{definition}[theorem]{Definition}
\newtheorem{lemma}[theorem]{Lemma}
\newtheorem{proposition}[theorem]{Proposition}

\newtheorem{remark}[theorem]{Remark}
\newtheorem{example}[theorem]{Example}

\title{The logic with unsharp implication and negation}
\author{Ivan~Chajda and Helmut~L\"anger}
\date{}

\begin{document}


\maketitle

\begin{abstract}
It is well-known that intuitionistic logics can be formalized by means of Brouwerian semilattices, i.e.\ relatively pseudocomplemented semilattices. Then the logical connective implication is considered to be the relative pseudocomplement and conjunction is the semilattice operation meet. If the Brouwerian semilattice has a bottom element $0$ then the relative pseudocomplement with respect to $0$ is called the pseudocomplement and it is considered as the connective negation in this logic. Our idea is to consider an arbitrary meet-semilattice with $0$ satisfying only the Ascending Chain Condition, which is trivially satisfied in finite semilattices, and introduce the connective negation $x^0$ as the set of all maximal elements $z$ satisfying $x\wedge z=0$ and the connective implication $x\rightarrow y$ as the set of all maximal elements $z$ satisfying $x\wedge z\leq y$. Such a negation and implication is ``unsharp'' since it assigns to one entry $x$ or to two entries $x$ and $y$ belonging to the semilattice, respectively, a subset instead of an element of the semilattice. Surprisingly, this kind of negation and implication, respectively, still shares a number of  properties of these connectives in intuitionistic logic, in particular the derivation rule Modus Ponens. Moreover, unsharp negation and unsharp implication can be characterized by means of five, respectively seven simple axioms. Several examples are presented. The concepts of a deductive system and of a filter are introduced as well as the congruence determined by such a filter. We finally describe certain relationships between these concepts. 
\end{abstract}



\section{Introduction}

Intuitionistic logic is usually algebraically formalized by means of {\em Brouwerian semilattices}, i.e.\ semilattices $(S,\wedge,*)$ where $*$ denotes relative pseudocomplementation which is considered as the connective implication, see \cite{B08}, \cite{B13}, \cite{Fr}, \cite K and \cite{M55}. If $(S,\wedge,*)$ has a $0$ then $x*0$ is the pseudocomplement of $x$ usually denoted by $x^*$ and considered as negation of $x$ in this logic. If $(S,\wedge,*,0)$ is even a lattice then it is called a {\em Heyting algebra}, see \cite H and \cite{M70}. For posets the concept of pseudocomplementation was extended and studied by the authors in \cite{C12} and \cite{CL21}.

It is well-known that every Brouwerian lattice (or Heyting algebra) is distributive.The concept of relative pseudocomplementation was extended by the first author to non-distributive lattices under the name sectional pseudocomplementation, see \cite{C03} and \cite{CLP}. Hence a kind of non-distributive intuitionistic logic can be created on sectionally pseudocomplemented lattices.

In their previous papers \cite{CL22a} and \cite{CLa} the authors showed that some important logics can be based also on posets that need not be lattices. An example of such a logic is the logic of quantum mechanics based on orthomodular posets, see e.g.\ \cite{CL22a}, \cite{FLP}, \cite{Fi} and \cite{PP}. It is evident that in this case some logical connectives such that disjunction or conjunction may be only partial operations or, as pointed out by the authors in \cite{CLa} and \cite{CL22b}, they may be be considered in an ``unsharp version'', i.e.\ their result need not be a single element but may be a subset of the poset in question. Thus also the connective implication is created in this way as ``unsharp''. For ``unsharpness'' see also \cite{GG}. This motivated us to study a variant of intuitionistic logic based on lattices that need neither be relatively pseudocomplemented nor even sectionally pseudocomplemented where the connective implication is unsharp.

\section{Preliminaries}

In the following we identify singletons with their unique element, i.e.\ we will write $x$ instead of $\{x\}$. Moreover, all posets considered in the sequel are assumed to satisfy the Ascending Chain Condition which we will abbreviate by ACC. This implies that every element lies under a maximal one. Of course, every finite poset satisfies the ACC. Let $(P,\leq)$ be a poset, $b\in P$ and $A,B\subseteq P$. By $\Max A$ we will denote the set of all maximal elements of $A$. We define
\begin{align*}
    A\leq B & \text{ if }a\leq b\text{ for all }a\in A\text{ and all }b\in B, \\
   A\leq_1B & \text{ if for every }a\in A\text{ there exists some }b\in B\text{ with }a\leq b, \\
A\approx_1B & \text{ if }A\leq_1B\text{ and }B\leq_1A.
\end{align*}
The relation $\leq_1$ is a quasiorder relation on $2^P$ and $\approx_1$ an equivalence relation on $2^P$. It is easy to see that $A\leq_1\Max B$ provided $A\subseteq B$ and that $A\leq_1b$ is equivalent to $A\leq b$.

Let $\mathbf S=(S,\wedge)$ be an arbitrary meet-semilattice and $A,B\subseteq S$. We define
\[
A\wedge B:=\{a\wedge b\mid a\in A,b\in B\}.
\]

\section{Unsharp negation}

Let $(S,\wedge,0)$ be a meet-semilattice with $0$ satisfying the ACC, $a\in S$ and $A\subseteq S$. We define
\[
a^0:=\Max\{x\in S|a\wedge x=0\}.
\]
Hence $^0$ is a unary operator on the meet-semilattice $(S,\wedge,0)$ with $0$ satisfying the ACC which assigns to every element $x\in S$ the non-void subset $x^0\subseteq S$. The element $a$ is called {\em sharp} if $a^{00}=a$. Moreover, we define
\[
A^0:=\Max\{x\in S|A\wedge x=0\}.
\]

We are going to prove the following properties of the operator $^0$ for every meet-semilattice with $0$ satisfying the ACC.

\begin{theorem}\label{th1}
Let $\mathbf S=(S,\wedge,0)$ be a meet-semilattice with $0$ satisfying the {\rm ACC} and $a,b\in S$. Then the following holds:
\begin{enumerate}[{\rm(i)}]
\item $a^0$ is an antichain,
\item $a\leq_1a^{00}$,
\item $a\leq b$ implies $b^0\leq_1a^0$,
\item $0^0=\Max S$,
\item $a\wedge a^0=0$,
\item if $\mathbf S$ is bounded then $0^0=1$ and $1^0=0$,
\item $a\wedge0^0\approx_1a$,
\item $a\wedge(a\wedge b)^0\approx_1a\wedge b^0$.
\end{enumerate}
\end{theorem}

\begin{proof}
\
\begin{enumerate}
\item[(i)] This is clear.
\item[(ii)] We have $a\in\{x\in S\mid a^0\wedge x=0\}$.
\item[(iii)] If $a\leq b$ then $\{x\in S\mid b\wedge x=0\}\subseteq\{x\in S\mid a\wedge x=0\}$.
\item[(iv)] and (v) follow directly from the definition of $a^0$.
\item[(vi)] If $\mathbf S$ is bounded then according to (iv)
\begin{align*}
0^0 & =\Max S=1, \\
1^0 & =\Max\{x\in S\mid1\wedge x=0\}=\Max\{0\}=0.
\end{align*}
\item[(vii)] According to (iv) we have $a\leq_1\Max S=0^0$ and hence $a\leq_1a\wedge0^0\leq a$.
\item[(viii)] Everyone of the following statements implies the next one:
\begin{align*}
       (a\wedge b)\wedge(a\wedge b)^0 & =0, \\
b\wedge\big(a\wedge(a\wedge b)^0\big) & =0, \\
                 a\wedge(a\wedge b)^0 & \leq_1b^0, \\
                 a\wedge(a\wedge b)^0 & \leq_1a\wedge b^0.
\end{align*}
From $a\wedge b\leq b$ we conclude $b^0\leq_1(a\wedge b)^0$ according to (iii) and hence $a\wedge b^0\leq_1a\wedge(a\wedge b)^0$.
\end{enumerate}
\end{proof}

From (iii) of Theorem~\ref{th1} there follows immediately $x^0\wedge y^0\leq(x\wedge y)^0$.

\begin{example}
Consider the meet-semilattice visualized in Fig.~1:

\vspace*{-3mm}

\begin{center}
\setlength{\unitlength}{7mm}
\begin{picture}(6,4)
\put(3,1){\circle*{.3}}
\put(1,3){\circle*{.3}}
\put(3,3){\circle*{.3}}
\put(5,3){\circle*{.3}}
\put(3,1){\line(-1,1)2}
\put(3,1){\line(0,1)2}
\put(3,1){\line(1,1)2}
\put(2.85,.3){$0$}
\put(.85,3.4){$a$}
\put(2.85,3.4){$b$}
\put(4.85,3.4){$c$}
\put(2.2,-.75){{\rm Fig.~1}}
\put(.7,-1.75){Meet-semilattice}

\end{picture}
\end{center}

\vspace*{8mm}

We have
\begin{align*}
                   a & =\{b,c\}^0=a^{00}, \\
	               0^0 & =\{a,b,c\}, \\
a\wedge(a\wedge b)^0 & =a\wedge0^0=a\wedge\{a,b,c\}=\{0,a\}=a\wedge\{a,c\}=a\wedge b^0
\end{align*}
in accordance with {\rm(ii)}, {\rm(iv)} and {\rm(viii)} of Theorem~\ref{th1}, respectively.
\end{example}

\begin{example}
Consider the modular lattice $\mathbf L$ depicted in Fig.~2:

\vspace*{-3mm}

\begin{center}
\setlength{\unitlength}{7mm}
\begin{picture}(10,8)
\put(3,1){\circle*{.3}}
\put(1,3){\circle*{.3}}
\put(3,3){\circle*{.3}}
\put(5,3){\circle*{.3}}
\put(7,3){\circle*{.3}}
\put(3,5){\circle*{.3}}
\put(5,5){\circle*{.3}}
\put(7,5){\circle*{.3}}
\put(9,5){\circle*{.3}}
\put(7,7){\circle*{.3}}
\put(3,1){\line(-1,1)2}
\put(3,1){\line(0,1)4}
\put(3,1){\line(1,1)2}
\put(3,1){\line(2,1)4}
\put(7,7){\line(-2,-1)4}
\put(7,7){\line(-1,-1)2}
\put(7,7){\line(0,-1)4}
\put(7,7){\line(1,-1)2}
\put(1,3){\line(2,1)4}
\put(3,3){\line(2,1)4}
\put(5,3){\line(2,1)4}
\put(1,3){\line(1,1)2}
\put(7,3){\line(1,1)2}
\put(5,3){\line(-1,1)2}
\put(7,3){\line(-1,1)2}
\put(2.85,.3){$0$}
\put(.35,2.85){$a$}
\put(2.35,2.85){$b$}
\put(4.35,2.85){$c$}
\put(7.4,2.85){$d$}
\put(2.35,4.85){$e$}
\put(5.4,4.85){$f$}
\put(7.4,4.85){$g$}
\put(9.4,4.85){$h$}
\put(6.85,7.4){$1$}
\put(4.2,-.75){{\rm Fig.~2}}
\put(2.8,-1.75){Modular lattice}
\end{picture}
\end{center}

\vspace*{8mm}

We have
\begin{align*}
       a^{00} & =\{g,h\}^0=a, \\
       f^{00} & =\{b,c\}^0=f, \\
a^0\wedge e^0 & =\{g,h\}\wedge d=d\neq\{g,h\}=a^0=(a\wedge e)^0.
\end{align*}
Hence $a$ and $f$ are sharp and the equality $x^0\wedge y^0=(x\wedge y)^0$ does not hold in general. In $\mathbf L$ from Figure~2 we have
\[
e^0=d\text{ and }d^0=e.
\]
Since $e\wedge d=0$ and $e\vee d=1$, $\{0,d,e,1\}$ is a complemented lattice.
\end{example}

If $a^0$ is a singleton, it need not be a complement of $a$, even if the semilattice is a lattice. E.g., consider the four-element lattice with atoms $a$ and $b$ and with an additional greatest element $1$. Then $a^0=b$, but $a\vee b\neq1$, i.e., $a^0$ is not a complement of $a$.

For every cardinal number $n$ let $\mathbf M_n=(M_n,\vee,\wedge)$ denote the bounded modular lattice of length $2$ having $n$ atoms.

The situation from Figure~2 can be generalized as follows.

\begin{remark}\label{rem1}
Every element of a direct product of a Boolean algebra and an arbitrary number of lattices $\mathbf M_n$ {\rm(}possibly different $n${\rm)} is sharp.
\end{remark}

This follows immediately from the fact that every element of a Boolean algebra and every element of the lattice $\mathbf M_n$ is sharp.

However, if the lattice $\mathbf L$ is not a direct product of two-element lattices and various $\mathbf M_n$ then the assertion of Remark~\ref{rem1} need not hold, see the following example.

\begin{example}
Consider the lattice visualized in Fig.~3:

\vspace*{-3mm}

\begin{center}
\setlength{\unitlength}{7mm}
\begin{picture}(6,10)
\put(3,1){\circle*{.3}}
\put(1,3){\circle*{.3}}
\put(3,3){\circle*{.3}}
\put(5,3){\circle*{.3}}
\put(1,5){\circle*{.3}}
\put(3,5){\circle*{.3}}
\put(1,7){\circle*{.3}}
\put(3,7){\circle*{.3}}
\put(3,9){\circle*{.3}}
\put(3,1){\line(-1,1)2}
\put(3,1){\line(0,1)8}
\put(3,1){\line(1,1)2}
\put(1,3){\line(0,1)4}
\put(1,3){\line(1,1)2}
\put(5,3){\line(-1,1)2}
\put(1,7){\line(1,1)2}
\put(2.85,.3){$0$}
\put(.35,2.85){$a$}
\put(3.4,2.85){$b$}
\put(5.4,2.85){$c$}
\put(.35,4.85){$d$}
\put(3.4,4.85){$e$}
\put(.35,6.85){$f$}
\put(3.4,6.85){$g$}
\put(2.85,9.4){$1$}
\put(2.2,-.75){{\rm Fig.~3}}
\put(2,-1.75){Lattice}
\end{picture}
\end{center}

\vspace*{8mm}

We have
\begin{align*}
              a^{00} & =\{b,c\}^0=f\neq a, \\
             a^{000} & =f^0=\{b,c\}=a^0, \\
              b^{00} & =\{c,f\}^0=b, \\
(a^0\wedge b^0)^{00} & =(\{b,c\}\wedge\{c,f\})^{00}=\{0,c\}^{00}=\{b,f\}^0=c\neq\{0,c\}=a^0\wedge b^0, \\
(c^0\wedge f^0)^{00} & =(\{b,f\}\wedge\{b,c\})^{00}=\{0,b\}^{00}=\{c,f\}^0=b\neq\{0,b\}=e^0\wedge f^0.
\end{align*}
Hence $a$ is not sharp, $b$ is sharp and the equality $(x^0\wedge y^0)^{00}=x^0\wedge y^0$ does not hold in general.
\end{example}

We are going to show that the operator $^0$ can be characterized by means of four simple conditions.

\begin{theorem}
Let $(S,\wedge,0)$ be a meet-semilattice with $0$ satisfying the {\rm ACC} and $^0$ a unary operator on $S$. Then the following are equivalent:
\begin{enumerate}[{\rm(i)}]
\item $x^0=\Max\{y\in S\mid x\wedge y=0\}$ for all $x\in S$,
\item the operator $^0$ satisfies the following conditions:
\begin{enumerate}[{\rm(P1)}]
\item $x^0$ is an antichain,
\item $x\wedge0^0\approx_1x$,
\item $x\wedge x^0\approx0$,
\item $x\wedge(x\wedge y)^0\approx_1x\wedge y^0$.
\end{enumerate}
\end{enumerate}
\end{theorem}

\begin{proof}
$\text{}$ \\
(i) $\Rightarrow$ (ii): \\
This follows from Theorem~\ref{th1}. \\
(ii) $\Rightarrow$ (i): \\
If $x\wedge y=0$ then according to (P2) and (P4) we have
\[
y\approx_1y\wedge0^0=y\wedge(x\wedge y)^0=y\wedge(y\wedge x)^0\approx_1y\wedge x^0\leq_1x^0
\]
and hence $y\leq_1x^0$. Conversely, if $y\leq_1x^0$ then according to (P3) we have
\[
x\wedge y\leq_1x\wedge x^0=0
\]
and hence $x\wedge y=0$. This shows that $x\wedge y=0$ is equivalent to $y\leq_1x^0$. We conclude
\[
\Max\{y\in S\mid x\wedge y=0\}=\Max\{y\in S\mid y\leq_1x^0\}=x^0.
\]
The last equality can be seen as follows. Let $z\in\Max\{y\in S\mid y\leq_1x^0\}$. Then $z\leq_1x^0$, i.e.\ there exists some $u\in x^0$ with $z\leq u$. We have $u\leq_1x^0$. Now $z<u$ would imply $z\notin\Max\{y\in S\mid y\leq_1x^0\}$, a contradiction. This shows $z=u\in x^0$. Conversely, assume $z\in x^0$. Then $z\leq_1x^0$. If $z\notin\Max\{y\in S\mid y\leq_1x^0\}$ then there would exist some $u\in S$ with $z<u\leq_1x^0$ and hence there would exist some $w\in x^0$ with $z<u\leq w$ contradicting (P1). This shows $z\in\Max\{y\in S\mid y\leq_1x^0\}$.
\end{proof}

\section{Unsharp implication}

Now we extend the operation of relative pseudocomplementation to arbitrary meet-semilattices with $0$ satisfying the ACC as follows: Let $\mathbf S=(S,\wedge,0)$ be a meet-semilattice with $0$ satisfying the {\rm ACC}, $a,b\in S$ and $A,B\subseteq S$. We define
\[
a\rightarrow b:=\Max\{x\in S\mid a\wedge x\leq b\}.
\]
Thus $\rightarrow$ is a binary operator on $\mathbf S$ assigning to every pair $(x,y)\in S^2$ the non-void subset $x\rightarrow y\subseteq S$. It is evident that
\[
x^0=x\rightarrow0\text{ for each }x\in S.
\]
Moreover, we define
\[
A\rightarrow B:=\Max\{x\in S\mid A\wedge x\leq B\}.
\]

\begin{example}
The ``operation table'' of the operator $\rightarrow$ in the meet-semilattice of Figure~1 looks as follows {\rm(}we write $abc$ instead of $\{a,b,c\}$ and so on{\rm)}:
\[
\begin{array}{c|cccc}
\rightarrow &  0  &  a  &  b  &  c \\
\hline
     0      & abc & abc & abc & abc \\
     a      & bc  & abc & bc  & ab \\
     b      & ac  & ac  & abc & ac \\
     c      & ab  & ab  & ab  & abc
\end{array}
\]
\end{example}

\begin{example}
The ``operation table'' of the operator $\rightarrow$ in the meet-semilattice of Figure~3 looks as follows {\rm(}we write $bc$ instead of $\{b,c\}$ and so on{\rm)}:
\[
\begin{array}{c|ccccccccc}
\rightarrow & 0  & a  & b  & c  & d  & e  & f  & g & 1 \\
\hline
     0      & 1  & 1  & 1  & 1  & 1  & 1  & 1  & 1 & 1 \\
		 a      & bc & 1  & bc & bc & 1  & 1  & 1  & 1 & 1 \\
		 b      & cf & cf & 1  & cf & cf & 1  & cf & 1 & 1 \\
		 c      & bf & bf & bf & 1  & bf & 1  & bf & 1 & 1 \\
		 d      & bc & g  & bc & bc & 1  & g  & 1  & g & 1 \\
		 e      & 0  & f  & b  & c  & f  & 1  & f  & 1 & 1 \\
		 f      & bc & g  & bc & bc & dg & g  & 1  & g & 1 \\
		 g      & 0  & f  & b  & c  & f  & ef & f  & 1 & 1 \\
		 1      & 0  & a  & b  & c  & d  & e  & f  & g & 1
\end{array}
\]
\end{example}

The following properties of the binary operator $\rightarrow$ can be proved.

\begin{theorem}\label{th2}
Let $\mathbf S=(S,\wedge,0)$ be a meet-semilattice with $0$ satisfying the {\rm ACC} and $a,b,c\in S$. Then the following holds:
\begin{enumerate}[{\rm(i)}]
\item $a\rightarrow b$ is an antichain,
\item $a\leq b$ implies $a\rightarrow b=\Max S$,
\item $b\in\Max S$ implies $b\in a\rightarrow b$,
\item $b\leq_1a\rightarrow b$,
\item $a\leq_1(a\rightarrow b)\rightarrow b$,
\item $a\leq b$ implies $c\rightarrow a\leq_1c\rightarrow b$ and $b\rightarrow c\leq_1a\rightarrow c$.
\item $a\wedge(a\rightarrow b)\approx_1a\wedge b$,
\item $a\rightarrow(b\wedge c)\approx_1(a\rightarrow b)\wedge(a\rightarrow c)$,
\item $(a\rightarrow b)\wedge b\approx_1b$,
\item if $\mathbf S$ is bounded then $1\rightarrow b=b$,
\item $a\wedge(b\rightarrow b)\approx_1a$,
\item if $\mathbf S$ is bounded then $a\rightarrow b=1$ if and only if $a\leq b$,
\item $b\leq_1a\rightarrow(a\wedge b)$.
\end{enumerate}
\end{theorem}

\begin{proof}
\
\begin{enumerate}
\item[(i)] This is clear.
\item[(ii)]\hspace*{-2mm}, (iv), (x), (xii) and (xiii) follow immediately from the definition of $\rightarrow$.
\item[(iii)] If $b\in\Max S$ then because of $a\wedge b\leq b$ we have $b\in\Max\{x\in S\mid a\wedge x\leq b\}=a\rightarrow b$.
\item[(v)] Since $a\wedge x\leq b$ for all $x\in a\rightarrow b$ we have $a\wedge(a\rightarrow b)\leq b$, i.e.\ $(a\rightarrow b)\wedge a\leq b$.
\item[(vi)] If $a\leq b$ then
\begin{align*}
\{x\in S\mid c\wedge x\leq a\} & \subseteq\{x\in S\mid c\wedge x\leq b\}, \\
\{x\in S\mid b\wedge x\leq c\} & \subseteq\{x\in S\mid a\wedge x\leq c\}.
\end{align*}
\item[(vii)] We have $a\wedge x\leq b$ and hence $a\wedge x\leq a\wedge b$ for all $x\in a\rightarrow b$ and hence $a\wedge b\leq_1a\wedge(a\rightarrow b)\leq a\wedge b$ according to (iv).
\item[(viii)] According to (vii) we have $a\rightarrow(b\wedge c)\leq_1(a\rightarrow b)\wedge(a\rightarrow c)$. Conversely, assume $d\in a\rightarrow b$ and $e\in a\rightarrow c$. Then $a\wedge d\leq b$ and $a\wedge e\leq c$ and hence $a\wedge(d\wedge e)\leq b\wedge c$ which implies $d\wedge e\leq_1a\rightarrow(b\wedge c)$. This shows $(a\rightarrow b)\wedge(a\rightarrow c)\leq_1a\rightarrow(b\wedge c)$.
\item[(ix)] We have $b\leq_1a\rightarrow b$ according to (iv) and hence $b\leq_1(a\rightarrow b)\wedge b\leq b$.
\item[(xi)] According to (ii) we have $a\leq_1\Max S=b\rightarrow b$ and hence $a\leq_1a\wedge(b\rightarrow b)\leq a$.
\end{enumerate}
\end{proof}

From Theorem~\ref{th2} it is evident that the binary operator $\rightarrow$ shares properties of the logical connective implication in intuitionistic logic despite the fact that it is {\em unsharp}, i.e.\ for $x,y\in S$ the result of $x\rightarrow y$ need not be a singleton. Hence it extends the intuitionistic logic based on a Heyting algebra $(L,\vee,\wedge,*,0)$ where again $\vee$ formalizes disjunction, $\wedge$ formalizes conjunction, but now $\rightarrow$ formalizes unsharp implication and $^0$ formalizes unsharp negation.

\begin{example}
Consider the modular lattice visualized in Fig.~4:

\vspace*{-3mm}

\begin{center}
\setlength{\unitlength}{7mm}
\begin{picture}(8,8)
\put(3,1){\circle*{.3}}
\put(1,3){\circle*{.3}}
\put(3,3){\circle*{.3}}
\put(5,3){\circle*{.3}}
\put(3,5){\circle*{.3}}
\put(5,5){\circle*{.3}}
\put(7,5){\circle*{.3}}
\put(5,7){\circle*{.3}}
\put(3,1){\line(-1,1)2}
\put(3,1){\line(0,1)4}
\put(3,1){\line(1,1)4}
\put(5,3){\line(-1,1)2}
\put(5,7){\line(-1,-1)4}
\put(5,7){\line(0,-1)4}
\put(5,7){\line(1,-1)2}
\put(2.85,.3){$0$}
\put(.35,2.85){$a$}
\put(2.35,2.85){$b$}
\put(5.4,2.85){$c$}
\put(2.35,4.85){$d$}
\put(5.4,4.85){$e$}
\put(7.4,4.85){$f$}
\put(4.85,7.4){$1$}
\put(3.2,-.75){{\rm Fig.~4}}
\put(1.9,-1.75){Modular lattice}
\end{picture}
\end{center}

\vspace*{8mm}

Then
\begin{align*}
                       e & \leq_1\{d,e\}=\{e,f\}\rightarrow e=(d\rightarrow e)\rightarrow e, \\
 d\wedge(d\rightarrow e) & =d\wedge\{e,f\}=c=d\wedge e, \\
(d\rightarrow e)\wedge e & =\{e,f\}\wedge e=\{c,e\}\approx_1e
\end{align*}
in accordance with {\rm(v)}, {\rm(vii)} and {\rm(ix)} of Theorem~\ref{th2}, respectively.
\end{example}

\begin{remark}
It is easy to see that the operation $\wedge$ and the operator $\rightarrow$ are related by so-called {\em unsharp adjointness}, i.e.
\[
a\wedge b\leq c\text{ if and only if }a\leq_1b\rightarrow c.
\]
Since the operation $\wedge$ is associative, commutative and monotone, it can be considered as a {\em t-norm}. Thus the semilattice $(S,\wedge,\rightarrow)$ endowed with the operator $\rightarrow$ is an {\em unsharply residuated semilattice}. Moreover, by {\rm(viii)} of Theorem~\ref{th2} we have
\[
a\wedge(a\rightarrow b)\approx_1a\wedge b
\]
showing that $(S,\wedge,\rightarrow)$ satisfies divisibility.
\end{remark}

Similarly as for the unary operator $^0$ we can characterize the binary operator $\rightarrow$ on a meet-semilattice with $0$ satisfying the ACC as follows.

\begin{theorem}
Let $(S,\wedge,0)$ be a meet-semilattice with $0$ satisfying the {\rm ACC} and $\rightarrow$ a binary operator on $S$. Then the following are equivalent:
\begin{enumerate}[{\rm(i)}]
\item $x\rightarrow y=\Max\{z\in S\mid x\wedge z\leq y\}$ for all $x,y\in S$,
\item The operator $\rightarrow$ satisfies the following conditions:
\begin{enumerate}[{\rm(R1)}]
\item $x\rightarrow y$ is an antichain,
\item $x\wedge(x\rightarrow y)\approx_1x\wedge y$,
\item $(x\rightarrow y)\wedge y\approx_1y$,
\item $x\rightarrow(y\wedge z)\approx_1(x\rightarrow y)\wedge(x\rightarrow z)$,
\item $x\wedge(y\rightarrow y)\approx_1x$,
\item $y\leq z$ implies $x\rightarrow y\leq_1x\rightarrow z$.
\end{enumerate}
\end{enumerate}
\end{theorem}

\begin{proof}
$\text{}$ \\
(i) $\Rightarrow$ (ii): \\
This follows from Theorem~\ref{th2}. \\
(ii) $\Rightarrow$ (i): \\
If $x\wedge z\leq y$ then according to (R3), (R5), (R4) and (R6) we have
\[
z\approx_1(x\rightarrow z)\wedge z\leq_1x\rightarrow z\approx_1(x\rightarrow x)\wedge(x\rightarrow z)\approx_1x\rightarrow(x\wedge z)\leq_1x\rightarrow y
\]
and hence $z\leq_1x\rightarrow y$. Conversely, if $z\leq_1x\rightarrow y$ then according to (R2) we have
\[
x\wedge z\leq_1x\wedge(x\rightarrow y)\approx_1x\wedge y\leq y
\]
and hence $x\wedge z\leq y$. This shows that $x\wedge z\leq y$ is equivalent to $z\leq_1x\rightarrow y$. We conclude
\[
\Max\{z\in S\mid x\wedge z\leq y\}=\Max\{z\in S\mid z\leq_1x\rightarrow y\}=x\rightarrow y.
\]
The last equality can be seen as follows. Let $u\in\Max\{z\in S\mid z\leq_1x\rightarrow y\}$. Then $u\leq_1x\rightarrow y$, i.e.\ there exists some $v\in x\rightarrow y$ with $u\leq v$. We have $v\leq_1x\rightarrow y$. Now $u<v$ would imply $u\notin\Max\{z\in S\mid z\leq_1x\rightarrow y\}$, a contradiction. This shows $u=v\in x\rightarrow y$. Conversely, assume $u\in x\rightarrow y$. Then $u\leq_1x\rightarrow y$. If $u\notin\Max\{z\in S\mid z\leq_1x\rightarrow y\}$ then there would exist some $v\in S$ with $u<v\leq_1x\rightarrow y$ and hence there would exist some $w\in x\rightarrow y$ with $u<v\leq w$ contradicting (R1). This shows $u\in\Max\{z\in S\mid z\leq_1x\rightarrow y\}$.
\end{proof}

\begin{example}
Consider the lattice from Figure~4. Then
\begin{align*}
 d\wedge(d\rightarrow e) & =d\wedge\{e,f\}=c=d\wedge e, \\
(d\rightarrow e)\wedge e & =e\wedge e=e, \\
 d\rightarrow(e\wedge f) & =d\rightarrow c=\{e,f\}\approx_1\{c,e,f\}=\{e,f\}\wedge\{e,f\}=(d\rightarrow e)\wedge(d\rightarrow f)
\end{align*}
in accordance with {\rm(R2)}, {\rm(R3)} and {\rm(R4)}, respectively.
\end{example}

\section{Deductive systems}

It is well-known that the connective implication in intuitionistic logic is closely related to the so-called deductive systems in the corresponding Brouwerian semilattice. In what follows we show that a certain modification of the concept of a deductive system plays a similar role for logics with unsharp implication. We define

\begin{definition}
A {\em deductive system} of a meet-semilattice $\mathbf S=(S,\wedge,0)$ with $0$ satisfying the {\rm ACC} is a subset $D$ of $S$ satisfying the following conditions for $x,y\in S$:
\begin{enumerate}[{\rm(D1)}]
\item $(\Max S)\cap D\neq\emptyset$,
\item $x\in D$ and $(x\rightarrow y)\cap D\neq\emptyset$ imply $y\in D$.
\end{enumerate}
\end{definition}

Recall that a {\em filter} of a meet-semilattice $\mathbf S=(S,\wedge)$ is a non-empty subset $F$ of $S$ satisfying the following conditions for $x,y\in S$:
\begin{enumerate}[{\rm(F1)}]
\item $x,y\in F$ implies $x\wedge y\in F$,
\item $x\in F$ and $x\leq y$ imply $y\in F$.
\end{enumerate}
It is clear that if $S$ is finite then all filters of $\mathbf S$ are given by the sets $[x):=\{y\in S\mid x\leq y\}$, $x\in S$, and hence the poset of all filters of $\mathbf S$ is dually isomorphic to $\mathbf S$ and therefore a join-semilattice where $[x)\vee[y)=[x\wedge y)$ for all $x,y\in S$.

For every non-empty subset $A$ of the universe of a meet-semilattice $(S,\wedge)$ we define a binary relation $\Theta(A)$ on $S$ as follows:
\[
(x,y)\in\Theta(A)\text{ if there exists some }a\in A\text{ with }x\wedge a=y\wedge a.
\]
Although the following result is known, for the reader's convenience we present the proof.

\begin{lemma}\label{lem1}
Let $\mathbf S=(S,\wedge,1)$ be a meet-semilattice with $1$ and $\Phi\in\Con\mathbf S$. Then the following holds:
\begin{enumerate}[{\rm(i)}]
\item $[1]\Phi$ is an filter of $\mathbf S$,
\item $\Theta([1]\Phi)\subseteq \Phi$.
\end{enumerate}
\end{lemma}

\begin{proof}
\
\begin{enumerate}[(i)]
\item
\begin{enumerate}[(F1)]
\item If $a,b\in[1]\Phi$ then $a\wedge b\in[1\wedge1]\Phi=[1]\Phi$.
\item If $a\in[1]\Phi$, $b\in S$ and $a\leq b$ then $b=1\wedge b\in[a\wedge b]\Phi=[a]\Phi=[1]\Phi$.
\end{enumerate}
This shows that $[1]\Phi$ is a filter of $\mathbf S$.
\item If $(a,b)\in\Theta([1]\Phi)$ then there exists some $c\in[1]\Phi$ with $a\wedge c=b\wedge c$ whence
\[
a=a\wedge1\mathrel\Phi a\wedge c=b\wedge c\mathrel\Phi b\wedge1=b
\]
which shows $(a,b)\in\Phi$.
\end{enumerate}
\end{proof}

Although our definition of a deductive system differs from that known for relatively pseudocomplemented semilattices, we are still able to prove the following relationships between the concepts mentioned before.

\begin{theorem}\label{th3}
Let $\mathbf S=(S,\wedge,0,1)$ be a bounded meet-semilattice satisfying the {\rm ACC} and $D$ a non-empty subset of $S$. Then the following are equivalent:
\begin{enumerate}[{\rm(i)}]
\item $D$ a deductive system of $\mathbf S$,
\item $D$ is an filter of $\mathbf S$,
\item $\Theta(D)\in\Con\mathbf S$ and $D=[1]\big(\Theta(D)\big)$.
\end{enumerate}
\end{theorem}

\begin{proof}
$\text{}$ \\
(i) $\Rightarrow$ (ii):
\begin{enumerate}
\item[(F2)] Assume $a\in D$, $b\in S$ and $a\leq b$. Then $a\rightarrow b=\Max S$ because of (ii) of Theorem~\ref{th2}. According to (D1) we have $(a\rightarrow b)\cap D=\Max S\cap D\neq\emptyset$ and hence $b\in D$ by (D2).
\item[(F1)] Let $a,b\in D$. Then by (xiii) of Theorem~\ref{th2} we have $b\leq_1a\rightarrow(a\wedge b)$. Hence there exists some $c\in a\rightarrow(a\wedge b)$ with $b\leq c$. Now (F2) implies $c\in D$ and therefore $\big(a\rightarrow(a\wedge b)\big)\cap D\neq\emptyset$ from which we conclude $a\wedge b\in D$ by (D2).
\end{enumerate}
(ii) $\Rightarrow$ (iii): \\
Evidently, $\Theta(D)$ is reflexive and symmetric. Let $(a,b),(b,c)\in\Theta(D)$. Then there exist $d,e\in D$ with $a\wedge d=b\wedge d$ and $b\wedge e=c\wedge e$. Because of (F1) we conclude $d\wedge e\in D$. Now
\[
a\wedge(d\wedge e)=(a\wedge d)\wedge e=(b\wedge d)\wedge e=(b\wedge e)\wedge d=(c\wedge e)\wedge d=c\wedge(d\wedge e)
\]
which yields $(a,c)\in\Theta(D)$, i.e.\ $\Theta(D)$ is transitive. Further, if $f\in S$ then
\[
(a\wedge f)\wedge d=(a\wedge d)\wedge f=(b\wedge d)\wedge f=(b\wedge f)\wedge d
\]
showing $(a\wedge f,b\wedge f)\in\Theta(D)$. Hence $\Theta(D)\in\Con\mathbf S$. If $a\in D$ then because of $a\wedge a=a=1\wedge a$ we have $a\in[1]\big(\Theta(D)\big)$ showing $D\subseteq[1]\big(\Theta(D)\big)$. Conversely, assume $a\in[1]\big(\Theta(D)\big)$. Then there exists some $b\in D$ with $a\wedge b=1\wedge b$. This implies $b\leq a$ wherefrom we conclude $a\in D$ by (F2) showing $[1]\big(\Theta(D)\big)\subseteq D$. \\
(iii) $\Rightarrow$ (i):
\begin{enumerate}[(D1)]
\item If $a\in D$ then, since $\mathbf S$ satisfies the ACC, there exists some $b\in\Max S$ with $a\leq b$ and hence
\[
b=1\wedge b\in[a\wedge b]\big(\Theta(D)\big)=[a]\big(\Theta(D)\big)=[1]\big(\Theta(D)\big)=D.
\]
\item If $a\in D$, $b\in S$ and $(a\rightarrow b)\cap D\neq\emptyset$ then there exists some $c\in D$ with $c\in a\rightarrow b$ and hence $a\wedge c\leq b$ whence
\[
b=1\wedge1\wedge b\in[a\wedge c\wedge b]\big(\Theta(D)\big)=[a\wedge c]\big(\Theta(D)\big)=[1\wedge1]\big(\Theta(D)\big)=[1]\big(\Theta(D)\big)=D.
\]
\end{enumerate}
\end{proof}

It is well known that for a filter $F$ of a relatively pseudocomplemented semilattice we have $(a,b)\in\Theta(F)$ if and only if $a\rightarrow b\in F$ and $b\rightarrow a\in F$. However, we can modify this result also for an arbitrary meet-semilattice with $0$ satisfying the ACC provided our unsharp implication is considered.

\begin{proposition}
Let $\mathbf S=(S,\wedge,0)$ be a meet-semilattice with $0$ satisfying the {\rm ACC}, $F$ a filter of $\mathbf S$ and $a,b\in S$. Then the following are equivalent:
\begin{enumerate}[{\rm(i)}]
\item $(a,b)\in\Theta(F)$,
\item $(a\rightarrow b)\cap F\neq\emptyset$ and $(b\rightarrow a)\cap F\neq\emptyset$.
\end{enumerate}
\end{proposition}

\begin{proof}
$\text{}$ \\
(i) $\Rightarrow$ (ii): \\
There exists some $c\in F$ with $a\wedge c=b\wedge c$. Hence $a\wedge c\leq b$ and $b\wedge c\leq a$ and therefore there exists some $d\in a\rightarrow b$ with $c\leq d$ and some $e\in b\rightarrow a$ with $c\leq e$. Because of (F2) we conclude $d,e\in F$ showing (ii). \\
(ii) $\Rightarrow$ (i): \\
Let $c\in(a\rightarrow b)\cap F$ and $d\in(b\rightarrow a)\cap F$. Then $c\wedge d\in F$ by (F1), $a\wedge c\leq b$ and $b\wedge d\leq a$. Hence
\[
a\wedge(c\wedge d)=(a\wedge c)\wedge(c\wedge d)\leq b\wedge(c\wedge d)=(b\wedge d)\wedge(c\wedge d)\leq a\wedge(c\wedge d),
\]
i.e.\ $a\wedge(c\wedge d)=b\wedge(c\wedge d)$ showing (i).
\end{proof}

{\bf Conclusion}

Although the implication within the logic based on the structure $(S,\wedge,0,\rightarrow)$ is unsharp, i.e.\ $x\rightarrow y$ may be a subset $I$ of $S$ which need not be a singleton, it has its logical meaning. Namely, we ask that $x\rightarrow y$ is the maximal element $c$ of $S$ satisfying $x\wedge c\leq y$ (where $\wedge$ denotes conjunction). And for each $c\in I$ this is satisfied. Moreover, the elements of $I$ are mutually incomparable. Thus we have no need to prefer one of them with respect to others. However, the expression
\[
x\wedge(x\rightarrow y)\leq y
\]
is nothing else than the derivation rule Modus Ponens (both in classical as well as in non-classical logic) since it properly says that the truth value of $y$ cannot be less than the truth value of the conjunction $x\wedge(x\rightarrow y)$ of $x$ and the implication $x\rightarrow y$. Hence, despite of the fact of unsharpness, such a logic is sound although it is derived from an arbitrary meet-semilattice with $0$ satisfying the ACC.

Authors' addresses:

Ivan Chajda \\
Palack\'y University Olomouc \\
Faculty of Science \\
Department of Algebra and Geometry \\
17.\ listopadu 12 \\
771 46 Olomouc \\
Czech Republic \\
ivan.chajda@upol.cz

Helmut L\"anger \\
TU Wien \\
Faculty of Mathematics and Geoinformation \\
Institute of Discrete Mathematics and Geometry \\
Wiedner Hauptstra\ss e 8-10 \\
1040 Vienna \\
Austria, and \\
Palack\'y University Olomouc \\
Faculty of Science \\
Department of Algebra and Geometry \\
17.\ listopadu 12 \\
771 46 Olomouc \\
Czech Republic \\
helmut.laenger@tuwien.ac.at
\end{document}